\documentclass[12pt]{amsart}
\usepackage{amssymb}
\usepackage{amscd}
\usepackage{amsthm}
\usepackage{eepic}
\usepackage{epic}

\usepackage{hyperref}

\swapnumbers

\newtheorem{Lemma}{Lemma}
\newtheorem{Theorem}[Lemma]{Theorem}

\newtheorem{Corollary}[Lemma]{Corollary}
\newtheorem {Claim}[Lemma]{Claim}
\theoremstyle{remark} 
\newtheorem{Remark}[Lemma]{Remark}

\theoremstyle{definition}

\newtheorem{Example}[Lemma]{Example}
\newtheorem{noTitle}[Lemma]{}

\numberwithin{Lemma}{section}

\def \C {{\mathbb C}}
\def \R {{\mathbb R}}
\def \T {{\mathbb T}}
\def \Z {{\mathbb Z}}
\def \N {{\mathbb N}}

\def \J {{\mathcal{J}}}

\def \M {{\mathcal{M}}}
\def \W {{\mathcal{W}}}

\def \calN {{\mathcal N}}

\def    \i {{\mathfrak{i}}}

\def    \t {{\mathfrak t}}

\def \CP {{\mathbb C}{\mathbb P}}
\def \tCP {{\mathbb C}{\mathbb P}{^2}}

\def    \oCP {{\mathbb C}{\mathbb P}{^1}}
\def \FS {{\text{FS}}}

\def \tDelta {{\tilde{\Delta}}}

\def \Inv {^{-1}}
\def \ssminus {\smallsetminus}
\def \eps {\epsilon}
\def \sss {\scriptstyle}
\def \Cinf {C^\infty}

\def \half {\frac{1}{2}}

\DeclareMathOperator \Diff {Diff}

\DeclareMathOperator \GL {GL}
\DeclareMathOperator \PSL {PSL}
\DeclareMathOperator \id {Id}
\DeclareMathOperator \Lie {Lie}
\DeclareMathOperator \AGL {AGL}
\DeclareMathOperator \image {image}

\DeclareMathOperator \perimeter {perimeter}
\DeclareMathOperator \area {area}

\DeclareMathOperator \Hirz {Hirz}

\def\eoe{\unskip\ \hglue0mm\hfill$\between$\smallskip\goodbreak}

\newcommand     {\comment}[1]   {}
\newcommand{\mute}[2] {}
\newcommand     {\printname}[1] {}
\newcommand{\labell}[1] {\label{#1}\printname{#1}}
\begin{document}

\title{Torus actions on small blow ups of $\tCP$}
\author{Liat Kessler}
\address{M.I.T. Department of Mathematics, Cambridge MA 02139, U.S.A}
\email{kessler@math.mit.edu}


\begin{abstract}
A manifold obtained by $k$ simultaneous symplectic blow-ups of $\CP^2$
of equal sizes $\epsilon$ 
(where the size of $\CP^1\subset\CP^2$ is one)
admits an effective two dimensional torus action if $k \leq 3$.
We show that it does not admit such an action if $k \geq 4$ and $\epsilon \leq \frac{1}{3k 2^{2k}}$.
For the proof, we correspond between the 
geometry of a symplectic toric four-manifold 
and the combinatorics of its moment map image. We also use techniques from the theory of J-holomorphic curves. 
\end{abstract}

\maketitle

\section{Introduction}
Let a torus $\T^{\ell} = (S^1)^{\ell}$ act effectively on 
a symplectic $2n$-dimensional manifold $(M,\omega)$. 
The action is called \emph{Hamiltonian}
if there exists a \emph{moment map}, that is, a map
$$ \Phi \colon M \rightarrow  {(\t^{\ell})}^* = {\R}^{\ell} $$  
that satisfies 
$$d \Phi_i = - \iota (\xi_i) \omega$$ 
for $i = 1,\ldots,{\ell}$, where $\xi_1,\ldots,\xi_{\ell}$ are the vector fields that 
generate the $\T^{\ell}$-action.
Unless said otherwise, we assume that $M$ is compact and connected.  
The image of the moment map, 
$$\Delta:= \Phi(M),$$ 
is then a convex polytope \cite{g-s:convexity}.

If $\dim {\T^{\ell}} = \half \dim M $, the triple $(M,\omega,\Phi)$ 
is a \emph{symplectic toric manifold}, and the torus action is called \emph{toric}.
The moment map image is  a \emph{Delzant polytope}; this means that   
the edges emanating from each vertex are generated 
by vectors $v_1,\ldots,v_n$ that span the lattice $\Z^n$.
By the Delzant theorem, $(M,\omega,\Phi)$ is determined by $\Delta$ up to an equivariant 
symplectomorphism.
Conversely, given a Delzant polytope $\Delta$ in $\R^n$, Delzant constructs a 
symplectic toric manifold $(M_{\Delta},\omega_{\Delta},\Phi_{\Delta})$  
whose moment map image is $\Delta$ \cite{delzant}.

As a result of Delzant's theorem and a combinatorial analysis of Delzant polygons, any 
symplectic toric four-manifold is obtained from either a standard $\tCP$ or a Hirzebruch surface, 
by a sequence of equivariant symplectic blow ups.  (See Lemma \ref{fulton}.)
However, it may be difficult to determine whether a given symplectic four-dimensional manifold 
is symplectomorphic to such a manifold.

For instance, let $(M_k,\omega_\epsilon)$ be a symplectic manifold
that is obtained from $(\tCP,\omega_{\FS})$ 
by $k$ simultaneous symplectic blow ups of equal sizes $\epsilon >0$.
(Our normalization convention for the Fubini-Study form $\omega_{\FS}$ is that the size of $\CP^1 \subset \CP^n$, $\frac{1}{2\pi} \int_{\CP^1} \omega_{\FS}$, is equal to one.) 
If $k \geq 4$, this manifold does not admit a toric action that is \emph{consistent with the blow ups},
 i.e., the blow ups cannot be performed equivariantly. 
(See Lemma \ref{T equivariant cor}.)  
Does it admit any other toric action?

In \cite{kk} we show that the answer is 'no' when $\epsilon$ is $1 \over n$ for an integer $n$. 
In this paper we show that the answer is 'no' for $\epsilon \leq \frac{1}{3k 2^{2k}}$, as a corollary of the following theorem.
\begin{Theorem}\label{main1}
If $(M_k,\omega_\epsilon)$ is symplectomorphic to $(M_\Delta,\omega_\Delta)$, for a Delzant polygon $\Delta$,
and 
\begin{equation} \label{epsssmall1}
 \epsilon \leq \frac{1}{3k 2^{2k}}, 
\end{equation}  
then $(M_\Delta,\omega_{\Delta},\Phi_{\Delta})$ can be obtained from $(\tCP,\omega_{\FS})$ 
by $k$ equivariant symplectic blow ups of equal size $\epsilon$.
\end{Theorem}

The theorem becomes false if we do not restrict $\epsilon$; 
for $\epsilon > \half$, there is a toric action on 
$(M_1,\omega_\epsilon)$ that is not consistent with the $\epsilon$-blow up, see Remark \ref{example}. 
Theorem \ref{main1} can be strengthen to the case $\epsilon \leq \frac{1}{3}$, see \cite[Corollary 3.14]{martin} and \cite[Theorem 3]{thesis}. However, here we use different methods in the proof; in particular, our arguments illustrate explicitly 
the behavior of $J_T$-holomorphic curves and their moment map-images. ($J_T$ denotes a $\T^2$-invariant complex structure on the manifold that is compatible with the symplectic form.) These novel arguments might be useful in other studies of torus actions on symplectic manifolds. We note that the proof in this paper is prior (and motivating) to the proof of the stronger claim. 

In proving Theorem \ref{main1}, we apply Gromov's compactness theorem for J-holomorphic curves to show the existence of 
$J_T$-curves in the homology classes of 
exceptional divisors 
obtained by the symplectic $\epsilon$-blow ups. 
In the case presented here, (as opposed to the case $\epsilon = {1 \over n}$ for an integer $n$), a-priori
these might be non-smooth cusp curves. We claim that in one of these homology classes there 
is a smooth $J_T$-holomorphic sphere. To prove this claim, we represent $J_T$-holomorphic spheres and cusp curves on the boundary of the moment map image, and 
reduce 
the claim to a combinatorial claim on the moment map polygon.  
 A key ingredient is Lemma \ref{push11}, saying that a $J_T$-holomorphic sphere whose moment map image avoids a neighbourhood of a vertex in the moment map polygon $\Delta$ can be pushed, by a gradient flow, to a connected union of preimages of a
chain of edges of $\Delta$.

The geometry-combinatorics correspondence is established in 
Section \ref{gecodata} and Section \ref{push}. 
The relevant results
from Gromov's theory of J-holomorphic curves are recalled in Section \ref{sec:holomorphic}.

To complete the proof of Theorem \ref{main1} by recursion, 
we need 
uniqueness of symplectic blow downs: symplectic blow downs along homologous curves
result in symplectomorphic manifolds. This is shown in the appendix.



\section{Reading geometric data from the moment map polygon}\label{gecodata}

\begin{noTitle}
An important model for a Hamiltonian action is $\C^n$ with the 
standard symplectic form, the standard $\T^n$-action given by rotations of the 
coordinates, and the moment map
\begin{equation} \labell{Cn mm}
 (z_1,\ldots,z_n) \mapsto \half (|z_1|^2,\ldots,|z_n|^2).
\end{equation}
The image of this moment map is the positive orthant, 
$$ \R_+^n = \{ (s_1,\ldots,s_n) \ | \ s_j \geq 0 \text{ for all $j$ } \}.$$

A Delzant polytope can be obtained by gluing open subsets of $\R_+^n$ 
by means of elements of $\AGL(n,\Z)$. ($\AGL(n,\Z)$ is the group of affine transformations
of $\R^n$ that have the form $x \mapsto Ax + \alpha$ 
with $A \in \GL(n,\Z)$ and $\alpha \in \R^n$.)
Similarly, a symplectic toric manifold can be obtained by gluing 
open $\T^n$-invariant subsets of $\C^n$ by means of 
equivariant symplectomorphisms and reparametrizations of $\T^n$. 
\end{noTitle}

\begin{noTitle} \labell{rational length}
The \emph{rational length} of an interval $d$ of rational slope in $\R^n$
is the unique number $\ell = |d|$ such that the interval 
is $\AGL(n,\Z)$-congruent to an interval of length $\ell$ 
on a coordinate axis. 
In what follows, intervals are always measured by rational length.
\end{noTitle}

\begin{noTitle} \labell{calJ}
An almost complex structure on a $2n$-dimensional manifold $M$ is an
automorphism of the tangent bundle, $J \colon TM \to TM$,
such that $J^2 = -\id$.  
It is \emph{compatible} with a symplectic form $\omega$ 
if $\langle u,v \rangle = \omega(u,Jv)$ is symmetric and positive definite. 
The \emph{first Chern class} 
of the symplectic manifold $(M,\omega)$ is defined to be the first Chern 
class of the complex vector bundle $(TM,J)$ and is denoted $c_1(TM)$.
This class is independent of the choice of 
compatible almost complex structure $J$ \cite[\S 2.6]{intro}.
\end{noTitle} 

\begin{Lemma} \labell{perimeter and area}
Let $(M,\omega)$ be a compact connected symplectic four-manifold.  
Let $\Phi \colon M \to \R^2$  be a moment map for a toric action,
and let $\Delta$ be its image.
\begin{enumerate}
\item 
The moment map-preimage of a vertex of $\Delta$ 
is a fixed point for the torus action, and the moment map-image of a fixed point 
is a vertex of $\Delta$.
\item \labell{i2}
Let $d$ be an edge of $\Delta$ of rational length $\ell$.  
Then its preimage, $\Phi\Inv(d)$, is a symplectically embedded 2-sphere 
in $M$ of symplectic area 
$$ \int_{\Phi\Inv(d)} \omega = 2 \pi \ell.$$
\item \labell{i3}
The (rational) perimeter of $\Delta$ is
$$ \perimeter(\Delta) = \frac{1}{2\pi} \int_M \omega \wedge c_1(TM) .$$ 
\item \labell{i4}
The area of $\Delta$ is
$$ \frac{1}{(2\pi)^2} \int_M \frac{1}{2!} \omega \wedge \omega .$$
\end{enumerate}
\end{Lemma}
For proof, see \cite[Lemma 2.10 and Lemma 2.2]{finite}.

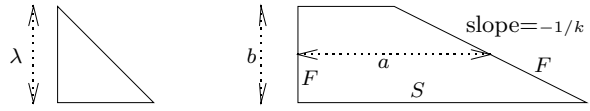
\begin{figure}[ht]
\setlength{\unitlength}{0.00083333in}
\begingroup\makeatletter\ifx\SetFigFont\undefined%
\gdef\SetFigFont#1#2#3#4#5{%
  \reset@font\fontsize{#1}{#2pt}%
  \fontfamily{#3}\fontseries{#4}\fontshape{#5}%
  \selectfont}%
\fi\endgroup%
{\renewcommand{\dashlinestretch}{30}
\begin{picture}(3612,639)(0,-10)
%
%
\path(300,612)(900,12)(300,12)(300,612)
\dottedline{45}(150,612)(150,12)
\path(120.000,132.000)(150.000,12.000)(180.000,132.000)
\path(180.000,492.000)(150.000,612.000)(120.000,492.000)
\put(0,260){\makebox(0,0)[lb]{\smash{{{\SetFigFont{12}{14.4}{\rmdefault}{\mddefault}{\updefault}$\scriptstyle{\lambda}$}}}}}
%
%
\path(1800,612)(2400,612)(3600,12)(1800,12)(1800,612)
%
%
\path(1920.000,342.000)(1800.000,312.000)(1920.000,282.000)
\dottedline{45}(1800,312)(3000,312)
\path(2880.000,282.000)(3000.000,312.000)(2880.000,342.000)
\put(2300,230){\makebox(0,0)[lb]{\smash{{{\SetFigFont{12}{14.4}{\rmdefault}{\mddefault}{\updefault}$\scriptstyle{a}$}}}}}
%
%
\dottedline{45}(1570,612)(1570,12)
\path(1600.000,492.000)(1570.000,612.000)(1560.000,492.000)
\path(1560.000,132.000)(1570.000,12.000)(1600.000,132.000)
\put(1480,260){\makebox(0,0)[lb]{\smash{{{\SetFigFont{12}{14.4}{\rmdefault}{\mddefault}{\updefault}$\scriptstyle{b}$}}}}}
\put(1820,120){\makebox(0,0)[lb]{\smash{{{\SetFigFont{12}{14.4}{\rmdefault}{\mddefault}{\updefault}$\scriptstyle{F}$}}}}}
%
%
\put(2850,450){\makebox(0,0)[lb]{\smash{{{\SetFigFont{12}{14.4}{\rmdefault}{\mddefault}{\updefault}\scriptsize{slope=}$\scriptstyle{-1/k}$}}}}}
%
%
\put(2500,50){\makebox(0,0)[lb]{\smash{{{\SetFigFont{12}{14.4}{\rmdefault}{\mddefault}{\updefault}$\scriptstyle{S}$}}}}}
\put(2025,650){\makebox(0,0)[lb]{\smash{{{\SetFigFont{12}{14.4}{\rmdefault}{\mddefault}{\updefault}$\scriptstyle{N}$}}}}}
\put(3275,200){\makebox(0,0)[lb]{\smash{{{\SetFigFont{12}{14.4}{\rmdefault}{\mddefault}{\updefault}$\scriptstyle{F}$}}}}}
\end{picture}
}
\caption{A Delzant triangle, $\Delta_\lambda$, 
         and a Hirzebruch trapezoid, $\Hirz_{a,b,k}$}
\labell{fig:triangle}
\end{figure}

\begin{Example} \labell{ex:Delzant}
Figure \ref{fig:triangle} shows examples of Delzant polygons 
with three and four edges.  
On the left there is a \emph{Delzant triangle},
$$ \Delta_\lambda = \{ (x_1,x_2) \ | \ 
   x_1 \geq 0 , x_2 \geq 0 , x_1 + x_2 \leq \lambda \}. $$
This is the moment map image of the standard toric action $(a,b) \cdot [z_0 : z_1 : z_2] = [z_0 : a z_1 : b z_2]$ on $\CP^2$, 
with the Fubini-Study symplectic form normalized so that 
the symplectic area of $\CP^1 \subset \CP^2$ is $2 \pi \lambda$.
The rational lengths of all its edges is $\lambda$.

On the right there is a \emph{Hirzebruch trapezoid},
$$ \Hirz_{a,b,k} = \{ (x_1,x_2) \ | \ 
   -\frac{b}{2} \leq x_2 \leq \frac{b}{2} , 0 \leq x_1 \leq a - kx_2 \} .$$
$b$ is the height of the trapezoid, $a$ is its average width, 
and $k$ is a non-negative integer such that the east edge has slope $-1/k$ 
or is vertical if $k=0$.  
We assume that $a \geq b$ and that $a - k \frac{b}{2} >  0$.
This trapezoid is a moment map image of a standard toric action 
on a Hirzebruch surface.
The rational lengths of its west and east edges are $b$;
the rational lengths of its north and south edges are $a \pm kb/2$.
\eoe
\end{Example}

\begin{noTitle} \labell{corner chopping}
Let $\Delta$ be a Delzant polytope in $\R^n$,
let $v$ be a vertex of $\Delta$,
and let $\delta > 0$ be smaller than the rational lengths
of the edges emanating from $v$. The edges of $\Delta$
emanating from $v$ have the form 
$\{ v + s \alpha_j \ | \ 0 \leq s \leq \ell_j \}$
where the vectors $\alpha_1, \ldots \alpha_n$ generate the lattice $\Z^n$
and $\delta < \ell_j$ for all $j$.
The \emph{corner chopping of size $\delta$} of $\Delta$ at $v$
is the polytope $\tDelta$ obtained from $\Delta$ by intersecting
with the half-space
$$ \{ \ v + s_1 \alpha_1 + \ldots + s_n \alpha_n \ \ | \  \ \ 
      s_1 + \ldots + s_n \geq \delta \} . $$
See, e.g., the chopping of the top right corner in Figure \ref{fig:blowup}.
The resulting polytope $\tDelta$ is again a Delzant polytope.
The corner chopping operation commutes with $\AGL(n,\Z)$-congruence:
if $\tDelta$ is obtained from $\Delta$ by a corner chopping
of size $\delta > 0$ at a vertex $v \in \Delta$ then, for any
$g \in \AGL(n,\Z)$, the polytope $g(\tDelta)$ is obtained
from the polytope $g(\Delta)$ by a corner chopping of size $\delta$
at the vertex~$g(v)$.  
\end{noTitle}

\begin{figure}[ht]
\setlength{\unitlength}{0.00083333in}
\begingroup\makeatletter\ifx\SetFigFont\undefined%
\gdef\SetFigFont#1#2#3#4#5{%
  \reset@font\fontsize{#1}{#2pt}%
  \fontfamily{#3}\fontseries{#4}\fontshape{#5}%
  \selectfont}%
\fi\endgroup%
{\renewcommand{\dashlinestretch}{30}
\begin{picture}(5211,1137)(0,-10)
\thinlines
\path(3044,912)(4544,912)(4844,612)
        (4844,12)(3044,12)(3044,912)
\thicklines
\path(3044,912)(4544,912)(4844,612)
\put(719,987){\makebox(0,0)[lb]{\smash{{{\SetFigFont{12}{14.4}{\rmdefault}{\mddefault}{\updefault}$\sss l_1$}}}}}
\path(2300,462)(2800,462)
\path(2750,512)(2800,462)(2750,412)
\thinlines
\path(44,912)(1844,912)(1844,12)
        (44,12)(44,912)
\put(1919,312){\makebox(0,0)[lb]{\smash{{{\SetFigFont{12}{14.4}{\rmdefault}{\mddefault}{\updefault}$\sss l_2$}}}}}
\put(4919,237){\makebox(0,0)[lb]{\smash{{{\SetFigFont{12}{14.4}{\rmdefault}{\mddefault}{\updefault}$\sss l_2-\delta$}}}}}
\put(794,387){\makebox(0,0)[lb]{\smash{{{\SetFigFont{12}{14.4}{\rmdefault}{\mddefault}{\updefault}$\sss \Delta$}}}}}
\put(3644,312){\makebox(0,0)[lb]{\smash{{{\SetFigFont{12}{14.4}{\rmdefault}{\mddefault}{\updefault}$\sss \Tilde{\Delta}$}}}}}
\put(3419,987){\makebox(0,0)[lb]{\smash{{{\SetFigFont{12}{14.4}{\rmdefault}{\mddefault}{\updefault}$\sss l_1-\delta$}}}}}
\put(4769,837){\makebox(0,0)[lb]{\smash{{{\SetFigFont{12}{14.4}{\rmdefault}{\mddefault}{\updefault}$\sss \delta$}}}}}
\end{picture}
}
\caption{Corner chopping} 
\labell{fig:blowup}
\end{figure}
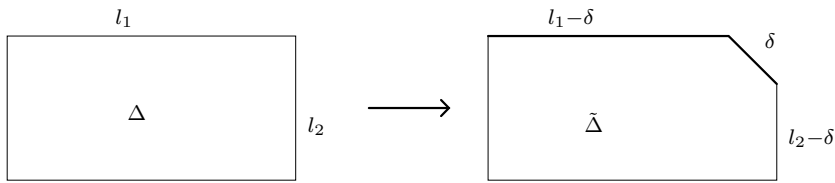

\begin{noTitle} \labell{blowup}
Recall that a blow up of \emph{size} $\eps=r^2/2$ of a $2n$-dimensional symplectic manifold $(M,\omega)$ 
is a new symplectic manifold $(\tilde{M},\tilde{\omega})$ 
that is constructed in the following way.
Let $\Omega \subset \C^n$ be an open subset that contains a ball 
about the origin of radius greater than $r$,
 and let 
$i \colon \Omega \to M$ be a symplectomorphism onto an open subset of $M$. (We consider $\C^n$ with the standard symplectic form.)
The standard symplectic blow up of $\Omega$
of size $r^2/2$ is obtained by removing
the open ball $B^{2n}(r)$ of radius $r$ about the origin 
and collapsing its boundary along the Hopf fibration 
$\partial B^{2n}(r) \to \CP^{n-1}$; the resulting space is naturally a smooth symplectic manifold \cite[Section 7.1]{intro}. 
This blow up transports to $M$
through $i$. 
The resulting copy of $(\CP^{n-1},\eps \omega_{\FS})$ in $\tilde{M}$ is called the 
\emph{exceptional divisor}.

If $M$ admits an action of a torus $\T^{\ell}$, and $i \colon \Omega \to M$ is 
$\T^{\ell}$-equivariant, where $\T^{\ell}$ acts on $\Omega$
through some homomorphism $\T^{\ell} \to U(n)$, 
then the torus action
naturally extends to the symplectic blow up of $M$ obtained from $i$, and the blow up
is \emph{equivariant}.
If the action on $M$ is Hamiltonian, its moment map naturally extends
to the blow up;  
in the case $\ell = n$ we call this a \emph{toric blow up}.

The moment map image of the standard symplectic blow up of $\C^n$
of size $\eps$  is obtained
from the moment map image $\R_+^n$ of $\C^n$ by corner chopping of size $\eps$. See Figure \ref{fig:BlC2} for $n=2$.

\begin{figure}
\setlength{\unitlength}{0.0006in}
\begingroup\makeatletter\ifx\SetFigFont\undefined%
\gdef\SetFigFont#1#2#3#4#5{%
  \reset@font\fontsize{#1}{#2pt}%
  \fontfamily{#3}\fontseries{#4}\fontshape{#5}%
  \selectfont}%
\fi\endgroup%
{\renewcommand{\dashlinestretch}{30}
\begin{picture}(4199,3300)(0,-10)
\dashline{60.000}(450,1050)(450,150)(1350,150)
\path(480.000,2905.000)(450.000,3025.000)(420.000,2905.000)
\path(450,3025)(450,1050)
\path(3405.000,120.000)(3525.000,150.000)(3405.000,180.000)
\path(3525,150)(1350,150)
\put(975,1350){\makebox(0,0)[lb]{\smash{{{\SetFigFont{12}{14.4}{\rmdefault}{\mddefault}{\updefault} $\sss |z_1|^2+|z_2|^2 \geq r^2$ }}}}}
\path(450,1050)(1350,150)
\put(225,525){\makebox(0,0)[lb]{\smash{{{\SetFigFont{12}{14.4}{\rmdefault}{\mddefault}{\updefault} $\sss \epsilon$ }}}}}
\put(0,3150){\makebox(0,0)[lb]{\smash{{{\SetFigFont{12}{14.4}{\rmdefault}{\mddefault}{\updefault} $\sss |z_2|^2/2$ }}}}}
\put(825,0){\makebox(0,0)[lb]{\smash{{{\SetFigFont{12}{14.4}{\rmdefault}{\mddefault}{\updefault} $\sss \epsilon$ }}}}}
\put(900,600){\makebox(0,0)[lb]{\smash{{{\SetFigFont{12}{14.4}{\rmdefault}{\mddefault}{\updefault} $\sss \epsilon$ }}}}}
\put(3600,75){\makebox(0,0)[lb]{\smash{{{\SetFigFont{12}{14.4}{\rmdefault}{\mddefault}{\updefault} $\sss |z_1|^2/2$ }}}}}
\end{picture}
}
\caption{Blow up of $\C^2$ of size $\displaystyle \eps = \frac{r^2}{2}$}
\labell{fig:BlC2}
\end{figure}

A toric
blow up of size $\eps$ of a symplectic toric manifold $(M,\omega,\Phi)$ at a fixed point $p$
amounts to chopping off a corner of size $\eps$ of its moment map image
$\Delta$ at the vertex $v=\Phi(p)$ to get a new polytope $\tilde{\Delta}$.
The preimage of the resulting new facet in $\tilde{\Delta}$ is the exceptional divisor 
in $\tilde{M}$. 
\end{noTitle}

We restrict our attention to symplectic toric manifolds of dimension $4$. 
Chopping off a corner of size $\eps$ of a polygon
$\Delta$ can be done if and only if there exist two adjacent edges
in $\Delta$ whose rational lengths are both strictly greater than $\eps$.
As a result, starting from a Delzant triangle of size $1$ 
we can perform one corner chopping of size $\epsilon > 0$  
if and only if $\eps < 1$, two or three corner choppings of size $\epsilon > 0$  if and only if $\eps < \half$, and no more than three corner choppings of the same size. Therefore,
\begin{Lemma} \labell{T equivariant cor}
$(\CP^2,\omega_\FS)$ admits a toric blow up of size $\eps > 0$ 
if and only if $\eps < 1$.
\ $(\CP^2,\omega_\FS)$ admits two or three toric blow ups 
of size $\eps > 0$ if and only if $\eps < \half$.
\ $(\CP^2,\omega_\FS)$ does not admit four or more toric blow ups 
of equal sizes.
\end{Lemma}
For a detailed proof, see \cite[Lemma 3.1]{kk}.

In $\R^2$, all Delzant polygons can be obtained 
by a simple recursive recipe:
\begin{Lemma} \labell{fulton}
\begin{enumerate}
\item \labell{f1}
Let $\Delta$ be a Delzant polygon with three edges. Then there exists 
a unique $\lambda > 0$ such that $\Delta$ is $\AGL(2,\Z)$-congruent
to the Delzant triangle $\Delta_\lambda$.
(See Example \ref{ex:Delzant}.)
\item \labell{f2}
Let $\Delta$ be a Delzant polygon with four or more edges.
Let $s$ be the non-negative integer such that the number of edges
is $4+s$.  Then there exist positive numbers $a \geq b > 0$, 
an integer $0 \leq k \leq 2a/b$, and positive numbers
$\delta_1, \ldots,\delta_s$, such that $\Delta$ is $\AGL(2,\Z)$-congruent 
to a Delzant polygon that is obtained from the Hirzebruch trapezoid 
$\Hirz_{a,b,k}$ 
(see Example \ref{ex:Delzant})
by a sequence of corner choppings of sizes $\delta_1, \ldots, \delta_s$.
\end{enumerate}
\end{Lemma}

\begin{proof}
See \cite[Section~2.5 and Notes to Chapter 2]{fulton}. 
\end{proof}

\begin{noTitle} \labell{Z of edges}
For any Delzant polygon $\Delta$, consider the free Abelian group
generated by its edges:
\begin{equation} \labell{edge group}
 \Z[\text{edges of }\Delta] .
\end{equation}

The ``length functional"
$$ \Z[\text{edges of }\Delta] \to \R $$
is the homomorphism that associates to each basis element its rational 
length.
If $\Delta_{i+1}$ is obtained from $\Delta_i$ by a corner chopping,
we consider the injective homomorphism 
\begin{equation} \labell{injection}
 \Z[\text{edges of }\Delta_{i}] \hookrightarrow
 \Z[\text{edges of }\Delta_{i+1}] 
\end{equation}
whose restriction to the generators is defined in the following way.
If $d$ is an edge of $\Delta$ that does not touch the corner
that was chopped, then $d$ is mapped to the edge of $\Delta_{i+1}$
with the same outward normal vector.  If $d$ is an edge of $\Delta_i$
that touches the corner that was chopped, then $d$ is mapped
to $d+e$ where $e$ is the new edge of $\Delta_{i+1}$, created
in the chopping.

The definition of corner chopping in \ref{corner chopping} 
implies that the homomorphism 
\eqref{injection} respects the length homomorphisms.
\end{noTitle}

By induction and the definition of corner chopping we get the following lemma.
\begin{Lemma} \labell{bound}
Let 
$$ \Delta_0,\Delta_1,\ldots,\Delta_s $$
be a sequence of Delzant polygons such that, for each $i$,
the polygon $\Delta_i$ is obtained from the polygon $\Delta_{i-1}$
by a corner chopping of size $\delta_i$.  Then
\begin{enumerate}
\item \labell{j4}
the image of an edge $d$ of $\Delta_j$ by $s-j$ iterations of homomorphism \eqref{injection} is a linear combination
$\sum_{i=0}^{\ell} m_i c_i$, such that $c_0,\ldots,c_{\ell}$ are edges of $\Delta_s$ whose union $U_d$ is connected, $\ell \leq (s-j)$, and 
for $0 \leq i \leq \ell$, the coefficient $m_i$ is a non-negative integer that is less than or 
equal to $2^{s-j}$;
we say that $d$ is \emph{given} 
by the chain $U_d$ with multiplicities $m_0,\ldots,m_{\ell}$. 
\item \labell{j1}
$ \area(\Delta_s) = \area(\Delta_0) 
                  - \half \delta_1^2 - \ldots - \half \delta_s^2 . $
\item \labell{j2}
$ \perimeter(\Delta_s) = \perimeter(\Delta_0) 
                       - \delta_1 - \ldots - \delta_s. $
\end{enumerate}
\end{Lemma}

\begin{Lemma} \label{possum}
Let $(M,\omega,\Phi)$ be a four-dimensional symplectic toric manifold, with moment-map polygon 
$\Delta$ of $n$ edges.
Then there are $n-2$ edges of $\Delta$ whose union is connected, such that the classes of their 
$\Phi$-preimages form a basis to 
$H_2(M;\Z)$. Moreover, for any $n-2$ edges of $\Delta$ whose union is connected,
 the classes of their preimages  form a basis to 
$H_2(M;\Z)$.
 \end{Lemma}

\begin{proof}
By Lemma \ref{fulton}, 
we can prove this by induction. 
In the induction step, 
suppose that $(\tilde{M}, \tilde{\omega},\tilde{\Phi})$ with moment map polygon 
$\tilde{\Delta}$ of $n+1$ edges is obtained by a toric blow up of $(M,\omega,\Phi)$ with moment map polygon $\Delta$. 
Let $B_{\Delta}$ be a set of  $n-2$ edges of $\Delta$ whose union is connected, such that the classes of their 
$\Phi$-preimages form a basis to 
$H_2(M;\Z)$. 
If $B_{\Delta}$ consists of an edge that touches the corner that was chopped, 
set $B_{\tilde{\Delta}}$ to be the edges of $\tilde{\Delta}$ with the same outward normal vector as the edges in  $B_{\Delta}$  plus the new edge $e$ of $\tilde{\Delta}$, created
in the chopping. If none of the edges in $B_{\Delta}$ touches the corner that was chopped, set $B_{\tilde{\Delta}}$ to be the edges of $\tilde{\Delta}$ with the same outward normal vector as the edges in  $B_{\Delta}$  plus one of the edges adjacent to $e$ in $\tilde{\Delta}$. 
\end{proof}

\begin{Corollary} \label{bett}
Let $(M,\omega,\Phi)$ be a four-dimensional symplectic toric manifold, with moment-map polygon 
$\Delta$.
The number of edges of $\Delta$ is equal to the second Betti number
$\dim H_2(M)$ plus two. 
\end{Corollary}

By the Delzant theorem, every toric action on $\CP^2$ is obtained
from a symplectomorphism of $\CP^2$ with a symplectic toric 
manifold $M_\Delta$ that is associated to a Delzant polygon $\Delta$.  
By Corollary \ref{bett}, $\Delta$ must be a triangle. 
By part \ref{f1} of Lemma \ref{fulton},
$\Delta$ is $\AGL(2,\Z)$-congruent
to a Delzant triangle $\Delta_\lambda$;
(see Example \ref{ex:Delzant}).
By our normalization convention for the Fubini-Study form, $\lambda=1$.
It follows that 
\begin{Lemma} \labell{toric is standard}
Every toric $\T^2$-action on $\CP^2$ is equivariantly symplectomorphic 
to the standard action. 
\end{Lemma}


\section{$J$-holomorphic spheres in symplectic $4$-manifolds} 
\labell{sec:holomorphic}

In this section we will highlight results from the theory of J-holomorphic curves that we will use for the proof 
of Lemma \ref{push11}, 
and to show uniqueness of symplectic blow downs in the appendix.

Let $(M,\omega)$ be a compact symplectic manifold. 
Let  $\J=\J(M,\omega)$ be the space of almost complex structures on $M$ 
that are 
compatible with $\omega$. The space $\J$ is contractible \cite{intro}. 
Given $J \in \J$,
a \emph{parametrized $J$-holomorphic sphere} is a map $u \colon \CP^1 \to M$, 
such that $du \colon T\CP^1 \to TM$
satisfies the Cauchy-Riemann equation $du \circ i = J \circ du$.
Such a $u$ represents a homology class in $H_2(M;\Z)$ that we denote $[u]$.
A $J$-holomorphic sphere 
is called \emph{simple} if it cannot be factored through a branched 
covering of $\CP^1$.
One similarly defines a holomorphic curve in $(M,J)$ whose domain is a Riemann 
surface other than $\CP^1$.

For any class $A\in H_{2}(M;\Z)$, consider the universal moduli space
of simple parametrized holomorphic spheres in the class $A$,
$$ \M(A,\J) = \{ (u,J) \ | \ J \in \J, 
 u \colon \CP^1 \to M \text{ is simple $J$-holomorphic, and } [u]=A \},$$ 
and the projection map
$$ p_{A} \colon \M(A,\J) \to \J. $$
For a fixed $J \in \J$, we denote by $\M(A,J)$ the space $p_{A}^{-1}(J)$.

The automorphism group $\PSL(2,\C)$ of $\CP^1$ 
acts on
$\M(A,\J)$ by reparametrizations.
The quotient $\M(A,\J) / \PSL(2,\C)$ is the space of unparameterized J-holomorphic spheres representing
$A \in H_2(M)$.

\begin{Lemma} \labell{frpr}
Let $0 \neq A \in H_2(M;\Z)$.
The action of $G=\PSL(2,\C)$ on $\M(A,\J)$ is
\begin{enumerate}
\item free, 
\item proper. 
\end{enumerate}
\end{Lemma}

\begin{proof}
For any sphere $u \in \M(A,\J)$, since $u$ is simple, the stabilizer $G_u=\{\psi \in G | u \circ \psi = u\}$ is trivial; this proves part (1). 

For part (2), we need to show that the action map $(u,\psi) \mapsto (u, u \circ \psi)$ is proper.
Let $K \subset \M(A,\J) \times \M(A,\J)$ be a compact subset. Without loss of generality 
$K = K_1 \times K_2$, where $K_1$ and $K_2$ are compact in $\M(A,\J)$. 
Because $\M(A,\J)$ is Hausdorff and first countable, \comment{do we need to justify this?}
it is enough to show that for every sequence
$\{ (u_n,\psi_n) \}$ in the preimage of $K_1 \times K_2$ there exists a subsequence
such that $\{\psi_n\}$ converges uniformly and $\{u_n\}$ converges in the $C^{\infty}$ topology.
Take such a sequence $\{u_n,\psi_n\}$. Because $u_n \in K_1$ and $K_1$ is compact, after passing to a subsequence we may assume that $\{u_n\}$ 
$C^\infty$-converges.

 By \cite[Lemma D.1.2]{nsmall}, if the sequence $\psi_n$ does not have a uniformly convergent subsequence, then there exist points $x,y \in \CP^1$ and a subsequence $\psi_{\mu}$ which converges to the point $y$ uniformly in compact subsets of $\CP^1 \setminus \{x\}$, in particular $\psi_{\mu}$ converges to a point on half sphere, hence $u_{\mu} \circ \psi_{\mu}$, restricted to half sphere, converge to a constant map. However, the sequence of holomorphic spheres $\{u_n \circ \psi_n\}$, (as a sequence in the compact subset $K_2$ of $\M(A,\J)$), has a $C^{\infty}$-convergent (hence u.c.s-convergent) 
subsequence whose limit is in the nontrivial homology class $A$; 
and we get a contradiction.
\end{proof}

Gromov, in \cite{gromovcurves}, introduced a notion of ``weak convergence"
of a sequence of holomorphic curves. This notion is preserved under
reparametrization of the curve, and it implies convergence in homology. 
\emph{Gromov's compactness theorem} guarantees that, given a converging
sequence of almost complex structures, a corresponding sequence 
of holomorphic curves with bounded symplectic area has a weakly converging
subsequence. 
The limit under weak convergence might not be a curve; it might
be a cusp curve, 
that is, a connected union of holomorphic curves.
As a result of Gromov's compactness, we have the following lemma.
\begin{Lemma}\labell{lem:gromov}
Let $\{ J_n \} \subset \J$ be a sequence of almost complex structures
that converges in the $C^\infty$ topology to an almost complex
structure $J_\infty \in \J$.  For each $n$, let $f_n \colon \CP^1 \to M$ 
be a parametrized $J_n$-holomorphic sphere.  Suppose that the set of areas 
$\omega([f_n])$ is bounded.  Then one of the following two possibilities
occurs.
\begin{enumerate}
\item
There exists a $J_\infty$-holomorphic sphere $u \colon \CP^1 \to M$ 
and elements $A_n \in \PSL(2,\C)$ such that a subsequence of the
$f_n \circ A_n $'s converges to $u$ in the $C^\infty$ topology.
In particular, there exist infinitely many $n$'s for which $[f_n] = [u]$.
\item
There exist two or more non-constant simple $J_\infty$-holomorphic
spheres $u_\ell \colon \CP^1 \to M$ and positive integers
$m_\ell$, for $\ell = 1, \ldots, L$, and infinitely many $n$'s for which
$$ [f_n] = \sum_{\ell=1}^L m_\ell [u_\ell] \qquad \text{ in } H_2(M).$$
\end{enumerate}
\end{Lemma}
For details, see \cite[Lemma A.3]{finite}.

In the proof of Lemma \ref{push11}, we will use the following Lemma.
\begin{Lemma} \labell{jexis}
Let $(M,\omega)$ be a closed symplectic four-manifold. Let $E \in H_2(M;\Z)$
be a homology class that can be represented by an embedded symplectic sphere and 
such that $c_1(TM)(E)  = 1$. 
Then for every $J \in \J$ there exists a $J$-holomorphic cusp curve 
in the class $E$.
\end{Lemma}

To deduce the lemma from Gromov's compactness we need the existence of a dense set $U \subset \J$ such that for any $J \in U$, the class $E$ is represented by an embedded $J$-holomorphic sphere. 

For any positive number $K$, let
$$ \calN_K = \{ A \in H_2(M;\Z) \ | \ 
 A \neq 0 , c_1(TM)(A) \leq 0 , \text{ and } \omega(A) < K \} . $$
 The importance of this set is in the fact that if a homology class $E$ with $\omega(E) \leq K$ and 
 $c_1(TM)(E) \leq 1$ is represented by a J-holomorphic cusp curve  with two or more components, then at least one of these
 components must lie in a homology class in $\calN_K$; see Lemma A.5 in \cite{finite}. 
Let
$$ U_K = \J \ssminus\bigcup_{A \in \calN_K} \image p_A .$$

Let $(M,\omega)$ be a compact symplectic four-manifold.
Then the subset $U_K \subset \J$ is open, dense, and path connected.
This is proved in \cite[Sec. 3]{McD:topology}
and \cite[Lemma~3.1]{McDuff-Structure}, 
and presented in \cite{finite}: Lemma A.8 and Lemma A.10. 
The following is also shown there.

\begin{Lemma} \labell{repdense}
Let $(M,\omega)$ be a compact symplectic four-manifold.
Let $E \in H_2(M)$ be a homology class that can be represented
by an embedded symplectic sphere and such that $c_1(TM)(E) = 1$.  Then
\begin{enumerate}
\item
The projection map $ p_E \colon \M(E,\J) \to \J $ is open.
\item
Let $K \geq \omega(E)$. Then, for any $J \in U_K$, 
the class $E$ is represented by an embedded $J$-holomorphic sphere. 
\end{enumerate}
\end{Lemma} 
For proof, see \cite[Lemma A.12]{finite}. 

Lemma \ref{jexis} now follows. 

For the proof of Theorem \ref{main1},
we also need the following Lemmas. 

\begin{Lemma} \labell{adjunction} 
Let $(M,\omega)$ be a compact symplectic four-manifold. 
Let $A \in H_2(M;\Z)$ be a homology class which is represented
by an embedded symplectic sphere $C$.  Then
\begin{enumerate}
\item
There exists an almost complex structure $J_0 \in \J$
for which $C$ is a $J_0$-holomorphic sphere.
\item
For any $J \in \J$ and any simple parametrized $J$-holomorphic sphere 
$f \colon \CP^1 \to M$ in the class $A$, the map $f$ is an embedding.
\end{enumerate}
\end{Lemma}

The lemma is a consequence of the adjunction formula. 
For details and references see, e.g., 
\cite[Lemma 5.3]{kk}.

\begin{Lemma} \labell{rock2}
Let $(M,\omega)$ be a compact symplectic four-manifold. 
Let $A \in H_2(M;\Z)$ be a homology class 
that is represented by an embedded symplectic sphere, and such that
$c_1(TM)(A) =1$. 
Let $J \in \image p_A$, and $(u,J) \in \M(A,J)$.

If $A=\sum_{i=1}^n{m_i [u_i]}$, 
where each component $u_i$ is 
a simple parametrized $J$-holomorphic sphere and $m_i \in \N$, then 
all the 
components  but one must be constants, and the nonconstant component
differs from $u$ by reparameterization of $\oCP$.
\end{Lemma}

\begin{proof}
By Lemma \ref{adjunction}, $u$ is an embedding, so the adjunction equality 
$$0=2+A \cdot A - c_1(TM)(A)$$ holds; since $c_1(TM)(A) =1$ this implies $A \cdot A = -1$.
If $n>1$ and there is more than one nonconstant component, 
then for $1 \leq i \leq n$, $\omega([u]) > \omega([u_i])$ so $u \neq u_i$, hence
by positivity of intersections of J-holomorphic spheres in a  four-manifold 
\cite[Theorem 2.6.3]{nsmall}, $[u_i] \cdot [u] \geq 0$.
Thus $0 \leq  \sum_{i=1}^n{m_i([u_i] \cdot [u])} =A \cdot A$, in contradiction to $A \cdot A =-1$.\\

Thus, all the components but one must be constants. By a similar argument, the nonconstant component 
differs from $u$ at most by reparameterization of $\oCP$.  
\end{proof}

\begin{Lemma} \labell{ue}
Let $(M,\omega)$ be a closed symplectic four-manifold. Let $E \in H_2(M;\Z)$
be a homology class that can be represented by an embedded symplectic sphere and such that 
$c_1(TM)(E) = 1$. 
Let $$U_E = \image p_E.$$ 
Then 
\begin{enumerate}
\item
$U_E \subset \J$ is open, dense, and path connected. Between any two elements in $U_E$ there is a path in $U_E$ that is transversal to $p_E$.
\item
The map
$$ \tilde{p_E} \colon \M(E,\J)/ \PSL(2,\C) \to U_E $$
that is induced from the projection map $p_E$ is proper.
\item 
For $J_0, J_1 \in U_E$, the sets $\M(E,J_0) / \PSL(2,\C)$ and $\M(E,J_1) / \PSL(2,\C)$
consist each of a single point, and there exists a path $\{J_t\}_{0 \leq t \leq 1}$ such that
$$\W(E;\{J_{t}\})=\{(u_{t},J_t) \ \mid \ u_t \in \M(E,J_{t}) \} / \PSL(2,\C)$$
is a compact one-dimensional manifold, and each $u_t$ is an embedding.
\end{enumerate}
\end{Lemma}

\begin{proof}
\begin{enumerate}
\item Since $p_E$ is an open map (part (1) of Lemma \ref{repdense}), 
its image $U_E$ is an open set in $\J$. 
Set $K=\omega(E)$.
By part (2) of Lemma \ref{repdense}, $U_K \subseteq U_E$. 
Since $U_K$ is dense in $\J$, so is $U_E$. 
Since $U_E$ is open, $\J$ locally path connected, and $U_K$ is dense in $U_E$ and path connected, 
we get that $U_E$ is path connected.

By the regularity criterion of Hofer-Lizan-Sikorav \cite{hls}, any element in $U_E$ is a regular value for $p_E$. A path between regular values for $p_E$ can be perturbed to a path with the same endpoints that is transversal to $p_E$; see \cite[Theorem 3.1.7(ii)]{nsmall}, see also \cite[Lemma A.9(d)]{finite}.

\item
 Part (2) follows from Gromov's compactness in the following way.
      Let $D \subset U_E$ be a compact subset. 
We need to show that $p_E\Inv(D)/\PSL(2,\C)$ is compact.
Because $\M(E,\J)$ is Hausdorff and first countable, \comment{do we need to justify this?}
it is enough to show that for every sequence
$\{ (f_n,J_n) \}$ in $p_E\Inv(D)$ there exists a subsequence
that, after reparametrization, has a limit in $p_E \Inv(D)$
in the $C^\infty$ topology.

Take such a sequence, $\{ (f_n,J_n) \} $.
Because $J_n \in D$ and $D$ is compact and contained in $U_E$, 
after passing to a subsequence we may assume that $\{J_n\}$ 
converges to $J_\infty \in U_E$.
Each $f_n$ is a $J_n$-holomorphic sphere in the class $E$.
Suppose that there exists a subsequence that, after reparametrization,
converges to some $u \colon \CP^1 \to M$ in the $\Cinf$ topology.
Then $u$ must be in the class $E$ and it must be $J_\infty$-holomorphic.
If $u$ is not simple, we get a contradiction to Lemma \ref{rock2}.
Then the pair $(u,J_\infty)$ is in the moduli space $\M(E,\J)$,
and since $J_\infty \in D$, this pair is in $p_E\Inv(D)$.

Now suppose that there does not exist such a subsequence.
By Lemma \ref{lem:gromov}, there exist two or more non-constant 
simple $J_\infty$-holomorphic spheres $u_\ell \colon \CP^1 \to M$
and positive integers $m_\ell$ such that $\sum m_\ell [u_\ell] = E$.
This contradicts Lemma \ref{rock2}.

 \item For $J \in U_E = \image p_E$, the set $\M(E,J)=p_E^{-1}(J) \neq \emptyset$.
Hence, by Lemma \ref{rock2}, the set $\M(E,J) / \PSL(2,\C)$ consists of a single point.
For $J_0, J_1 \in U_E$, by part (1), there is a path $\{J_t\}$ in $U_E$ from $J_0$ to $J_1$, that is
transversal to $p_E$,
hence, by  \cite[Theorem 3.1.7]{nsmall}, $\W^{*}(E;\{J_{t}\}) =\{(u_{t},J_t) \ \mid \ u_t \in \M(E,J_{t})\}$ is a manifold
of dimension $1+6=1+\text{index}p_E$.
By Lemma \ref{frpr}, the action of $\PSL(2,\C)$ on $\W^{*}(E;\{J_{t}\})$ is free and proper,
thus 
$$\W(E;\{J_{t}\})=\{(u_{t},J_t) \ \mid \ u_t \in \M(E,J_{t})\} / \PSL(2,\C) $$ is a manifold of dimension one. $\W(E;\{J_{t}\})$ is the inverse image of the path $\{J_t\}$
under the map $\tilde{p_E}$, hence, by part (2), it 
is compact.

By Lemma \ref{adjunction}, 
each $u_t$ is an embedding.
\end{enumerate}
\end{proof}

\section{Representing $J_T$-holomorphic curves on the moment map polygon}\label{push}

{\bf Notation:}
For $(M,\omega,\Phi)$, let $J_T$ denote  a $\T^n$-invariant complex structure
on $M$ that is compatible with $\omega$. By Delzant's construction \cite{delzant}, such a structure exists. 

\begin{Claim} \label{JI}
Let $(M,\omega,\Phi)$ be a four-dimensional symplectic toric manifold, with moment-map polygon 
$\Delta$. 
The preimage under $\Phi$ of an edge $d$ of $\Delta$ is an embedded  
$J_T$-holomorphic sphere.
\end{Claim}

\begin{proof}
By part \ref{i2} of Lemma \ref{perimeter and area}, $Y=\Phi\Inv(d)$ is a symplectically embedded 2-sphere 
in $M$. 
Being a connected component of a fixed point set of a holomorphic $S^1$-action, $TY = J_T TY$. 
As an almost complex manifold of real dimension two, $(Y,J_T|_{TY})$ is a complex manifold. 
Thus the embedded sphere $Y$ is 
an embedded holomorphic sphere in the complex manifold $(M,J_T)$.
\end{proof}

\begin{Lemma} \label{possum2}
Let $(M,\omega,\Phi)$ be a four-dimensional symplectic toric manifold, with moment-map polygon 
$\Delta$.
Then
\begin{itemize}
\item 
any $J_T$-holomorphic sphere 
is homologous in $H_2(M;\Z)$ to 
a linear combination with coefficients in $\N$ of inverse images under $\Phi$ of  
edges of $\Delta$;  \\
\item 
for any set $S$ of $n-2$ edges whose union is connected, 
any simple $J_T$-holomorphic sphere $C$ that is not the preimage of an edge of $\Delta$ is homologous 
to a linear combination with coefficients in $\N$ of
 preimages of edges of $\Delta$ whose union is connected and that are contained in $S$; if 
the intersection of $C$ with each of the two edges of $\Delta$ that are not in $S$ is positive, 
then 
all the $n-2$ edges of $S$ appear with positive coefficients in this linear combination. 
\end{itemize}
\end{Lemma}

\begin{proof}
\begin{itemize}
\item
Let $\Psi$ be an $S^1$-moment map obtained by composing $\Phi$ with projection in a rational direction
along which there is not any edge of $\Delta$.
Denote by $v_{\min}$ ($v_{\max}$) the vertex of minimal (maximal) value of that projection. 
Let $D_1,\ldots,D_m$ be a chain of $\Phi$-preimages of edges  between $v_{\min}$ and $v_{\max}$. 
Let  ${D'}_1,\ldots,{D'}_{m'}$ be the other chain of $\Phi$-preimages of edges between $v_{\min}$ and $v_{\max}$.

Without loss of generality we assume that $C$ is a simple $J_T$-holomorphic sphere 
that is not the $\Phi$-preimage of an edge of $\Delta$. 
By Lemma \ref{possum}, in $H_2(M;\Z)$ 
\begin{equation} \label{sumeq}
[C]=\sum_{i=1}^{m}a_i D_i + \sum_{j=1}^{m'}b_j {D'}_j,\text{ with }a_1=b_1=0.
\end{equation}
Adapting the proof of Lemma C.6 in \cite{karshon} we get that
\begin{equation} \label{chain}
\text{for }1 \leq i<m \quad (1 \leq i <m'), \quad {a_{i+1} / k_{i+1}} \geq {a_i / k_i} \geq 0,
\end{equation}
where $k_i$ is the order of the stabilizer of the $i$-th sphere in a chain. 

Notice that (\ref{chain}) implies that 
$$a_{\ell} >0 \Rightarrow a_i >0 \text{ for all }\ell \leq i \leq m,$$
and 
$$b_{\ell} >0 \Rightarrow b_j >0 \text{ for all }\ell \leq j \leq m'.$$ 
Hence, $C$ is homologous in $H_2(M;\Z)$ to 
a linear combination with coefficients in $\N$ of inverse images under $\Phi$ of,  
at most $n-2$, edges of $\Delta$ whose union is connected.  \\

\item It is enough to observe that for any set $S$ of $n-2$ edges whose union is connected, 
there is an $S^1$-moment map $\Psi$,  
obtained by composing $\Phi$ with projection in a rational direction 
along which there is not any edge of $\Delta$, such that the vertex $v_{\min}$ is 
the vertex  between the two edges of $\Delta$ that are not in $S$.
Then the previous proof gives the required.

\end{itemize}
\end{proof}

\begin{Lemma} \label{push11}
Let $(M,\omega,\Phi)$ be a four-dimensional symplectic toric manifold with moment-map 
polygon $\Delta$. Let $J_T$ be a $\T^2$-invariant $\omega$-compatible complex structure on $M$, and $g_T$ be the  
Riemannian metric defined by $(\omega,J_T)$. 
Let ${\i}^*$  be a projection in a rational direction
along which there is not any edge of $\Delta$.
Let $v_{\min}$ be the vertex of $\Delta$ of minimal value of that projection.

Let $C$ be a $J_T$-holomorphic sphere such that 
$\Phi(C)$ avoids the vertex $v_{\min}$.
Let $\alpha$ and $\beta$ be the points 
of $\Phi(C)$ on the boundary of $\Delta$, that are closest to $v_{\min}$ from left and right. 
Let $v_{\alpha}$ and $v_{\beta}$ be the vertices following $\alpha$ and $\beta$.
Then the gradient flow of $\Psi={\i}^{*}\circ{\Phi} $ with respect to $g_T$ carries $C$ to a family of $J_T$-holomorphic spheres; 
this family weakly converges 
to a connected union of preimages of edges of $\Delta$ (maybe with multiplicities). 
These edges form a chain that we denote $L_C$. The vertices of $L_C$ closest to $v_{\min}$ from left and right are 
$v_\alpha$ and $v_\beta$.
\end{Lemma}

\begin{proof} 
The function $\Psi={\i}^{*}\circ{\Phi} \colon M \to \R$ is a moment map associated with   
a Hamiltonian action on $(M,\omega)$ of 
$S^1$ embedded in $\T^2$ by $\i \colon S^1 \hookrightarrow \T^2.$ 
 
Let $\xi_M$ be the vector field generating the $S^1$-action.
The gradient flow $\eta_t$ of $\Psi$ with respect to the invariant metric $g_T$ is generated by $-J_T \xi_M$. 
This flow is equivariant with respect to the action, i.e., for each $t$, 
the diffeomorphism $\eta_t \colon M \rightarrow M$ 
is $\T^2$-equivariant.
Consequently, it sends a set that is a $\Phi$-preimage of a vertex or a $\Phi$-preimage of 
an edge to itself.

Set $L$ to be the chain of edges of $\Delta$ that do not touch $v_{\min}$.
Let $B=\{p \in M \colon i^{*} \circ \Phi(p) > r\}$ for some $i^{*}(v_{\min}) < r < \min \{ i^{*}(v'),i^{*}(v'') \}$, where $v'$ ($v''$) is the vertex following $v_{\min}$ immediately from the left (right). 
Then $\cap_{t > 0}(\eta_t(B)) \supseteq  \Phi^{-1}L$.
On the other hand, a point $p \in B$ that is not in $\Phi^{-1}(L)$, is sent to $v_{\min}$ by the gradient flow $\eta_t$ as $t \to -\infty$, i.e., for $t'$ big enough, $q=\eta_{-t'}(p)$ is not in $B$. Since $\eta_{t'}$ is a diffeomorphism, there cannot be $b \in B$ such that $\eta_{t'}(b)=\eta_{t'}(q)=\eta_{0}(p)=p$, in particular, $p$ is not in the intersection $\cap_{t > 0}(\eta_t(B))$. 
So  $$\cap_{t >0}(\eta_t(B)) = \Phi^{-1}(L).$$

We choose $B$ big enough such that, for some complex coordinates, the complexified toric action 
on $M-B$ is the standard action of the complex torus on an open subset of ${\C}^2$. 
In particular, for $t_1,t_2$ close to $0$, if $t_1 > t_2 >0$, $\eta_{t_1}(M-B) \supset \eta_{t_2}(M-B)$, hence $\eta_{t_1}(B) \subset \eta_{t_2}(B)$. Since $\eta_t$ is a flow, 
(i.e., a homomorphism from $(\R,+)$ to $(\Diff, \circ)$), 
this implies that for any $t_1 > t_2 > 0$, $\eta_{t_1}(B) \subset \eta_{t_2}(B)$, i.e., $\eta_t$ is monotonic on $B$.

Now, choose $B$ such that, in addition to the above, its image contains $\Phi(C)$. 
Consider a sequence $\{C_i\}$, where $C_i=\eta_i(C)$, with discrete $i \rightarrow \infty$. 
Each $C_i$ is a $J_T$-holomorphic sphere in the homology class $[C]$. 
By Gromov's compactness theorem, 
there is a subsequence $\{C_{\mu}\}$ that weakly converges to a $J_T$-holomorphic 
(maybe non-smooth) cusp curve $C'$ in $[C]$. 
In particular, each point in the limit $C'$ is the limit of a sequence of points in $\{C_{\mu}\}$, hence,
since $C_{\mu}=\eta_{\mu}(C) \subset \eta_{\mu}(B)$, and $\eta_t$ is monotonic on $B$, 
we get that $C' \subset  \cap_{\mu}(\eta_{\mu}(B)) \subset \Phi \Inv (L)$.
Thus, since
each edge preimage 
is an irreducible 
$J_T$-holomorphic sphere in the complex manifold $(M,J_T)$ (by Claim \ref{JI}), the irreducible components of $C'$ are preimages of edges in $L$. 
We conclude that the cusp curve $C'$ is a connected union of preimages of the edges of a subchain $L_C$ of $L$, 
with positive multiplicities.

Let $p_{\alpha}$ ($p_{\beta}$) be the preimage of $v_\alpha$ ($v_\beta$) in $M$.
The chain $L_C$ includes $v_\alpha$ and $v_\beta$, as the limits of $\eta_{\mu}(p_\alpha)$ and $\eta_{\mu}(p_\beta)$.
Assume a vertex $v$ on $L_C$ closer to $v_{\min}$ from the left than $v_{\alpha}$.
Let $e_v$ be the edge that touches $v$ from below. Then $L_C$ intersects $e_v$ at $v$, hence $\Phi^{-1}(L_C)$ intersects $\Phi^{-1}(e_v)$ at the point $\Phi^{-1}(v)$, maybe with multiplicities. However $C \cap \Phi^{-1}(e_v)= \emptyset$, in contradiction to $[\Phi^{-1}(L_C)]=[C]$. 
Similarly, the vertex on $L_C$ closest to $v_{\min}$ from right is $v_{\beta}$.
\end{proof}

\begin{Claim} \label{push2}
Let $(M,\omega,\Phi)$ be a four-dimensional symplectic toric manifold with moment-map 
polygon $\Delta$.

Every $J_T$-cusp curve 
$C$ is homologous in $H_2(M;\Z)$ to a linear combination with coefficients in $\N$ 
of preimages of edges of $\Delta$ 
whose union is connected. In particular, $C$ is homologous to a $\T^2$-invariant  $J_T$-cusp curve.
\end{Claim}

We already know that a $J_T$-cusp curve $C$ is homologous to a linear combination with coefficients in $\N$ 
of preimages of edges of $\Delta$, (by applying the first part of Lemma \ref{possum2} to the components of the cusp curve). However, the union of these edges might not be connected. The 'connected' part that we add here plays an important role in the proof of Theorem \ref{main1}.

\begin{proof}
Let ${\i}^{*}$ be a projection in a rational direction along which there is not any edge of $\Delta$. 
Let $v_{\min}$ be the vertex of $\Delta$ of minimal value of ${\i}^{*}$.
If for any component of $C$ that is not a $\Phi$-preimage of an edges of $\Delta$, 
the moment map image avoids a neighbourhood of $v_{\min}$, 
then the claim follows from Lemma \ref{push11} (and the fact that $C$ is connected). 
Otherwise, 
there is such a component $D$;
by positivity of intersections, the intersection number of $D$ with the preimage 
of each of the two edges adjacent to $v_{\min}$ is positive. 
Thus, by the second part of Lemma \ref{possum2}, 
$D$ is homologous to a linear combination with 
coefficients in $\N$ of $\Phi$-preimages of all the edges of $\Delta$ but the two 
adjacent to $v_{\min}$.
By the first part of Lemma \ref{possum2}, each component  of $C$ is homologous to a linear combination
of $\Phi$-preimages of edges of $\Delta$ 
with coefficients in $\N$. 
Combining such representatives of $D$ and the other components of $C$ gives the claim. 
\end{proof}

\begin{Lemma} \labell{sympr}
Let $(M,\omega,\Phi)$ be a symplectic toric four-manifold with moment map polygon $\Delta$. 
Let $C$ be an embedded symplectic sphere in $(M,\omega)$ with $c_1(TM)(C)=1$.

Then $C$ is homologous in $H_2(M;\Z)$ to a linear combination with coefficients in $\N$ 
of preimages of edges of $\Delta$ 
whose union is connected.
\end{Lemma}

\begin{proof}
By Lemma \ref{jexis} there exists a $J_T$-holomorphic cusp curve 
in the class $[C]$.
Now apply Claim \ref{push2}.
\end{proof}

\section{No toric action on $(M_k, \omega_{\epsilon})$ for $k >3$ and small $\epsilon$}
\labell{results}

For $\epsilon >0$, let
$$ (M_k,\omega_\epsilon) $$
denote a symplectic manifold that is obtained from $(\tCP,\omega_{\FS})$ 
by $k$ simultaneous 
symplectic blow ups
of equal sizes $\epsilon$. For description of symplectic blow up, see \S \ref{blowup}. 
$k$ \emph{simultaneous blow ups} are obtained from embeddings
$i_1 \colon \Omega_1 \to M$, $\ldots,$ $i_k \colon \Omega_k \to M$
whose images are disjoint.
we denote by $E_1,\ldots,E_k$ the homology classes in $H_{2}(M_k;\Z)$ 
of the exceptional divisors  obtained by the blow ups, 
and by $L$ the homology class of a line $\oCP \subset M_k$.

\begin{noTitle} \labell{bu}
By McDuff and Polterovitch, 
 for $k \leq 8$ there exists a symplectic blow up of $\CP^2$ $k$ times by size $\epsilon$ if and only if $\eps$ satisfies the following conditions. If $k=2,3,4$: $\eps < \half$. If $k=5,6$: $\eps < \frac{2}{5}$. If $k=7$: $\eps < \frac{3}{8}$. If $k=8$: $\eps < \frac{6}{17}$. See \cite{mystery}. According to Biran, for $k \geq 9$, there exist $k$ symplectic blow-ups of equal sizes $\eps$ if and only if $\eps$ satisfies the volume constraint, i.e., $\eps < \frac{1}{\sqrt{k}}$. See \cite{biran}. 
\end{noTitle}

Assume that $(M_k,{\omega}_{\epsilon})$ admits a toric action
with moment map polygon $\Delta$.
By Lemma \ref{sympr}, 
 each $E_i$ can be represented by a 
linear combination with coefficients in $\N$ of preimages 
of edges of $\Delta$. We call the union of these edges, with the $\N$-multiplicities, 
a \emph{$\Delta$-representative} of $E_i$. 
If this union is connected,
we call it a \emph{connected $\Delta$-representative}.
We observe the following properties of $\Delta$-representatives of $E_1,\ldots,E_k$.

\begin{Claim} \label{pos}
Assume that $(M_k,{\omega}_{\epsilon})$ admits a toric action
with moment map image $\Delta$. 
Choose $\Delta$-representatives for $E_1,\ldots,E_k$.  
For $m \leq k$, the number of edges in the union of the $\Delta$-representatives of
$m$ different $E_i$'s is $> m$, 
unless each of these $\Delta$-representatives is a single edge with multiplicity one.
\end{Claim}

\begin{proof}
Assume that the union of the chosen $\Delta$-representatives of $E_1,\ldots,E_m$ (without loss of generality) is
a subset of 
the set of edges $C_1,\ldots,C_m$, i.e., 
in  $H_{2}(M_k;\Z)$, for $1 \leq i \leq m$,
 \begin{equation} \label{sum0}
 E_i= \sum_{j=1}^{m} a_{j}^{i}[{\Phi}^{-1}C_j], \quad a_j^i \in {0} \cup \N.
\end{equation}
Denote by $A$ the $m \times m $ matrix of the coefficients $a_j^i$. 
Since the homology classes $E_1,\ldots,E_m$ are independent, the matrix $A$ is invertible (over $\R$).  
We get that 
\begin{equation} \label{sum}
([{\Phi}^{-1}C_1], \ldots,[{\Phi}^{-1}C_m])^t= A^{-1}(E_1,\ldots,E_m)^t. 
\end{equation}
 Since the homology classes $L,E_1,\ldots,E_k$ form a basis of 
 $H_{2}(M_k;\Z)$, each $[\Phi^{-1}C_j]= d_j L + \sum_{i}{{b_{i}^{j}}{E_i}}$, with unique integers as coefficients. 
The coefficients do not change if we write $[\Phi^{-1}{C_j}]$ as a linear combination of $L,E_i$ 
in $H_{2}(M_k;\R)$.
By this and \eqref{sum}, 
all the entries of $A^{-1}$ are in $\Z$, so in $H_{2}(M_k;\Z)$,
\begin{equation} \label{sum2}
[\Phi^{-1}C_j]= \sum_{i=1}^{m}{{b_{i}^{j}}{E_i}}, \quad b_i^j \in \Z. 
\end{equation}   
Since the size of each $E_i$ is $\epsilon$ we deduce that the length $|C_j|$ of each 
$C_j$ is an integer multiple of $\epsilon$. 
Since $|C_j| > 0$, it must be a
multiple of $\epsilon$ by $N_j \in \N$. 
However, by \eqref{sum0}, for $1 \leq i \leq m$, 
\begin{equation} \label{sum3}
 \epsilon= \sum_{j=1}^{m} {a_{j}^{i}|C_j|}, \quad a_j^i \in {0} \cup \N. 
\end{equation}
Thus
\begin{equation} \label{sum4}
 \epsilon= \sum_{j=1}^{m} a_{j}^{i}N_j \epsilon, \quad a_j^i \in {0} \cup \N, \  N_j \in \N.
\end{equation}
We get that 
in each line (and each column) of (the invertible matrix) $A$ there is $1$ 
in one entry and $0$ in each of the other entries, i.e., 
each of the $\Delta$-representatives is a single edge with multiplicity one.
\end{proof}

\begin{Claim} \label{poscon}
Assume that $(M_k,{\omega}_{\epsilon})$ admits a toric action 
with moment map image $\Delta$. 
Choose connected $\Delta$-representatives for $E_1,\ldots,E_k$. 
Denote their union by $U$. 
If none of the chosen connected $\Delta$-representatives is a single edge of 
$\Delta$ with multiplicity one, then      $U$
is connected and consists of at least $k+1$ edges.
\end{Claim}

\begin{proof}
By Claim \ref{pos}, 
$U$ consists of more than $k$ edges. 
Assume 
that $U$ is disconnected. Then it consists of at most $k+1$ edges, 
hence
it consists of exactly $k+1$ edges out of the $k+3$ edges of $\Delta$. 
    Since none of the $\Delta$-representatives is a single edge, Claim \ref{pos} implies that
 the $m_j$ edges of a connected component $j$ 
support at most $m_j-1$ of the $E_i$'s. 
Thus the non connected $k+1$ edges support at most $\sum_{j=1}^{c}{(m_j-1)} = k +1 - c < k$ 
of these classes, where $c>1$ is the number of connected components, and we get a contradiction.
\end{proof}

For a convex polygon $\Delta$ in $\R^2$, we denote by 
$$ (M_\Delta,\omega_\Delta,\Phi_\Delta) $$
a symplectic toric manifold 
whose moment map image is $\Delta$.

The main ingredient of the proof of Theorem \ref{main1} is:  
\begin{Claim} \label{fin}
If $(M_k,\omega_\epsilon)$ is symplectomorphic to $(M_\Delta,\omega_\Delta)$, 
and $$ \epsilon \leq \frac{1}{3k 2^{2k}},$$
then one of the classes $E_1,\ldots,E_k$  is realized by 
an embedded $\T^2$-invariant symplectic exceptional sphere,
 equivariantly 
blowing down along it yields
$(M_{k-1}, \omega_{\epsilon})$
with a toric action. 
\end{Claim}

\begin{proof} 
If $k \geq 1$, the moment map image $\Delta$ is a Delzant polygon of $k+3 \geq 4$ edges, 
so by Lemma \ref{fulton}, up to $\AGL(2,\Z)$-congruence, 
it  is obtained by $(k-1)$ corner-choppings
of sizes $(\delta_1,\ldots,\delta_{k-1})$ from 
a standard  Hirzebruch trapezoid $\Sigma$  
with west and east edges $F_w$, $F_e$, south edge $S$, north edge $N$, and slope  $-1/d$.

By part \eqref{j4} of Lemma \ref{bound}, 
\begin{equation} \label{boundperi}
|S|+|N| < 2^k \perimeter(\Delta), 
\end{equation}
and $F_w$ and $F_e$ are given by two disjoint connected unions of edges of $\Delta$ with multiplicities $\leq 2^k$.

For each class $E_i$, 
we choose a connected $\Delta$-representative, 
that is a connected union of edges (with multiplicities in $\N$) whose preimage is in $E_i$. 
Assume that none of these $\Delta$-representatives is a single edge of $\Delta$ with multiplicity one. 
By Claim \ref{poscon}, the union $U$ of these $\Delta$-representatives 
  is connected and consists of  
at least $k+1$ edges
of the $k+3$ edges of $\Delta$.
Then, (at least) one of the two chains of edges giving $F_w$ and $F_e$ as above
is contained in $U$: 
the 
connected at most two edges that are not in $U$ can overlap at most one chain giving $F_w$ or $F_e$, 
 since the two chains are separated at each end by an edge. 
Thus 
\begin{equation} \label{boundeps}
|F|=|F_w|=|F_e| \leq 2^k k \epsilon.
\end{equation}

Then
\begin{eqnarray*}
\half (1 - k{{\epsilon}^2}) =  \area (\Delta )
&=&  \half (|S|+|N|)|F|- \sum_{i=1}^{k-1}{\half{\delta_i}^2}\\
&\leq & \half 2^k  \perimeter (\Delta )|F| - \sum_{i=1}^{k-1}{\half{\delta_i}^2}\\
&=& \half 2^k (3-k \epsilon )|F| - \sum_{i=1}^{k-1}{\half{\delta_i}^2}\\
&\leq &  \half 2^k(3-k\epsilon )2^k k \epsilon - \sum_{i=1}^{k-1}{\half{\delta_i}^2}.
\end{eqnarray*}

The first (in)equality is by part (\ref{i4}) of Lemma \ref{perimeter and area}, 
the second is by part \eqref{j1} of Lemma \ref{bound}, 
the third  inequality is by equation (\ref{boundperi}), the fourth follows from part (\ref{i3}) of Lemma 
\ref{perimeter and area} and the fact that the Poincare dual to ${c_1}(T{M_k})$ equals $3L -\sum_{i=1}^k{E_i}$, and the last is by equation (\ref{boundeps}).

We get that
\begin{eqnarray*}
1 - k{{\epsilon}^2}  
&\leq & 2^{2k}(3  - k \epsilon) (k \epsilon)\\
&\leq & 2^{2k} (3 k \epsilon - k {\epsilon}^2). 
\end{eqnarray*} 
So, $1 \leq 2^{2k}3k\epsilon - k \epsilon^2 (2 ^{2k}-1)$, thus
$1 < 3k  2^{2k}   \epsilon $, in contradiction to the assumption on $\epsilon$.
Therefore, (at least) one of the classes $E_1,\ldots,E_k$ is represented by the 
inverse image under the moment map of a single edge of $\Delta$ with multiplicity one. 
 By Claim \ref{JI}, 
such a representative $C_T$ is an embedded $J_T$-holomorphic sphere.
It is $\T^2$-invariant: let $a \in \T^2$;  because $\T^2$ is connected, $[a C_T] = [C_T]$;
by positivity of intersections and since $E_i \cdot E_i = -1$,
$aC_T$ and $C_T$ must coincide.
Because $C_T$ is an embedded $J_T$-sphere and $J_T$ is compatible with $\omega_\epsilon$,
$C_T$ is symplectic.

Without loss of generality, the class $E_1$ is represented by such a $J_T$-holomorphic sphere $C_T$.
Set $J_0$ to be an almost complex structure on $(M_k,\omega_\epsilon)$ 
for which 
the exceptional divisors obtained by the symplectic blow ups
are disjoint embedded $J_0$-holomorphic spheres $S_1,\ldots,S_k$ 
that represent the classes $E_1,\ldots,E_k$. 
(Such a structure exists by Lemma \ref{adjunction}.) 
By Lemma \ref{gap} in the Appendix, the symplectic manifold resulting from 
$(M_k,\omega_\epsilon)$ by 
blowing down along $C_T$ is symplectomorphic to the symplectic manifold obtained by
 blowing down along $S_1$, which is $(M_{k-1},\omega_{\epsilon})$.
\end{proof}

\begin{proof}[Proof of Theorem \ref{main1}] 
Assume that $(M_k,\omega_\epsilon)$ is symplectomorphic to $(M_\Delta,\omega_\Delta)$ 
and 
$\epsilon \leq \frac{1}{3k 2^{2k}}$. 
After $k$ iterations of Claim \ref{fin}, 
we get $\tCP$ with a toric action.  
By Lemma \ref{toric is standard}, 
this manifold is equivariantly symplectomorphic to $\CP^2$ 
with its standard toric action.
By reversing our steps we get $\tCP$ 
blown up equivariantly $k$ times by equal sizes $\epsilon$.
\end{proof}

\begin{Remark} \label{example}
Theorem \ref{main1} becomes false if we do not restrict $\epsilon$. 
For $\epsilon > \half$,  
let $(M_1,\omega_{\epsilon},\Phi_1)$ be $\tCP$ blown up equivariantly by size $\epsilon$. 
The moment map image is obtained by chopping off a corner of size $\epsilon$ from a Delzant triangle of edge-size $1$, to get  
a trapezoid $\Hirz_{(1+\epsilon) /2,1-\epsilon,1}$, i.e.,  
of height $(1-\epsilon)$, average width $(1+\epsilon) /2$, and slope $-1$.
Let $(N,\omega_2,\Phi_2)$ be a Hirzebruch surface whose image is a trapezoid 
$\Hirz_{(1+\epsilon) /2,1-\epsilon,3}$
(Notice that the north edge is then of size  $2 \epsilon -1$, which is $>0$ if and only if $\epsilon > \half$.)
See Figure \ref{exam}. 
Since these Hirzebruch trapezoids have the same average width and height and the inverse 
of their slopes differ by $2$,
the corresponding manifolds  
are isomorphic as symplectic manifolds with Hamiltonian $S^1$-action (by \cite[Lemma 3]{YK:max_tori}), 
however they are not isomorphic as symplectic toric manifolds (their 
moment map polygons are not equivalent). 

  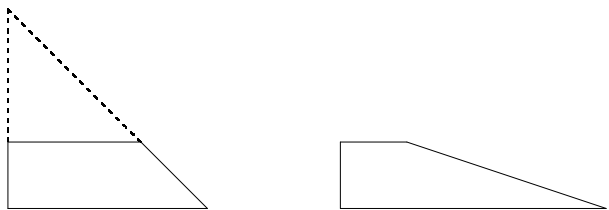
\begin{figure}
$$
\setlength{\unitlength}{0.00029in}
\begingroup\makeatletter\ifx\SetFigFont\undefined%
\gdef\SetFigFont#1#2#3#4#5{%
  \reset@font\fontsize{#1}{#2pt}%
  \fontfamily{#3}\fontseries{#4}\fontshape{#5}%
  \selectfont}%
\fi\endgroup%
{\renewcommand{\dashlinestretch}{30}
\begin{picture}(10824,3639)(0,-10)
\path(12,12)(3612,12)(12,12)
\path(12,12)(12,1212)(12,12)
\path(2412,1212)(3612,12)(2412,1212)
\path(6012,12)(9612,12)(6012,12)
\path(9612,12)(10812,12)(9612,12)
\path(6012,12)(6012,1212)(6012,12)
\path(10812,12)(7212,1212)(10812,12)
\path(12,1212)(2412,1212)(12,1212)
\path(6012,1212)(7212,1212)(6012,1212)
\dashline{90.000}(12,1212)(12,3612)(12,1212)
\dashline{90.000}(12,3612)(2412,1212)(12,3612)
\end{picture}
}
$$
\caption{symplectomorphic but not equivariantly symplectomrophic symplectic toric manifolds}
\label{exam}
\end{figure}
\end{Remark}

Theorem \ref{main1} and Lemma \ref{T equivariant cor} yield the following
corollary.

\begin{Corollary} \label{three1}
$(M_k, \omega_{\epsilon})$ with $\epsilon \leq \frac{1}{3k 2^{2k}}$ 
admits a toric action if and only if $k \leq 3$. 
\end{Corollary}
We notice that by the sharpness of the constrains listed in \S \ref{bu}, 
when $\epsilon \leq \frac{1}{3k 2^{2k}}$ there exists a 
symplectic blow up of $\tCP$ $k$ times by size $\epsilon$.

Since $H^{1}(M_k, \R)=\{0\}$, 
any effective $(S^1)^2$-action on $(M_k,\omega_\epsilon)$ is toric. 
\begin{Corollary} 
$(M_k, \omega_{\epsilon})$ with $\epsilon \leq \frac{1}{3k 2^{2k}}$ 
admits an effective $(S^1)^2$-action if and only if  $k \leq 3$. 
\end{Corollary}

\appendix
\section{Uniqueness of blow down} \label{appendix}

\begin{Lemma} \label{gap}
Let $(M,\omega)$ be a compact four-dimensional symplectic manifold. 
Let $J_0,J_1 \in \J$.
Let $A$ be a class in $H_2(M;\Z)$ such that $c_1(TM)(A)=1$ and $\omega(A)>0$. 
Assume that $A$ is represented by an embedded $J_0$-holomorphic sphere $C_0$ and by an embedded $J_1$-holomorphic 
sphere $C_1$.

Then for $i=0,1$, there are neighbourhoods $U_i$ of $C_i$, each symplectomorphic to a tubular neighbourhood 
of $\CP^1$, 
and a symplectomorphism $\phi$ of $(M,\omega)$, that sends $(U_0,C_0)$ to $(U_1,C_1)$, 
and induces the identity map on $H_2(M;\Z)$.  
\end{Lemma}

\begin{proof}
By part (3) of Lemma \ref{ue}, there is a smooth family (with parameter $0 \leq t \leq 1$) of 
$J_t$-holomorphic embeddings $\rho_t$ from $\oCP$ to the manifold.
Their images are 
all in the homology class $A$.
Notice that the pullbacks of $\omega$ to $\oCP$ by the homotopic maps are 
all in the same cohomology class. 
Hence, by Moser, there is a family of diffeomorphisms $\phi_t \colon \oCP \rightarrow \oCP$,    
starting at the identity map, that satisfy $\phi_t^{*}{({\rho}_0^*(\omega))}=\rho_t^*(\omega)$.
Hence we may assume that $\rho_0$ is a symplectic embedding of the standard $\oCP$ and 
compose the embeddings $\{\rho_t\}$ on the family $\{\phi_t\}$ to
get a one-parameter family 
of symplectic embeddings 
of 
the standard $\oCP$
into $M$.  
Moreover, using a parameterized version of Weinstein's tubular neighbourhood theorem, this family can be extended to
a one-parameter family 
of symplectic embeddings 
$\sigma_t$ 
of a neighbourhood 
of  
$\oCP$ (as the zero-section) in 
the tautological bundle  with a symplectic form, 
into $M$; denote the image of $\sigma_t$ by $U_t$.

We get a ``partial flow'' 
that moves along the neighbourhoods $U_t$.
Differentiating it by $t$ gives vector fields $X_t$, defined at $U_t$.
The Lie derivative $\Lie_{X_t}{\omega}$ is $0$. 
By Cartan's formula, 
$$ \Lie_{X_t}{\omega} =  d  (\iota_{{X}_t} \omega) + (\iota_{{X}_t}) d \omega = 
d \iota_{{X}_t} \omega,$$
where the last equality is since $\omega$ is closed.
Thus the one form $\iota_{X_t} \omega $ on $U_t$
is closed. 
Therefore, and since $\CP^1$ is simply connected, when we consider $X_t$ as a vector field defined at a neighbourhood of $\oCP \times [0,1] \subseteq M \times [0,1]$, 
we get a function $h$ defined on a (maybe smaller) neighbourhood of $\oCP \times [0,1] \subseteq M \times [0,1]$,
such that 
$\iota_{{X}_t} \omega = dh _t$.  
Using partition of unity in $M \times [0,1]$, we expand $h$ to a smooth function ${H} \colon M \times [0,1] \to \R$,  
whose restriction to a small neighborhood of  $\image \rho_t$
coincides with $h_t$.

This gives a Hamiltonian flow on $M$, thus a family of symplectomorphisms $\{\alpha_t\}_{0 \leq t \leq 1}$, 
starting from the identity map. Take $\alpha_1$ to be $\phi$. 
\end{proof}


\begin{thebibliography}{KKP}






\bibitem[Bi]{biran}
P.\ Biran, \emph{Symplectic packing in dimension 4},
 Geom.\ Funct.\  Anal.\ {\bf  7} (1997), no. 3, 420--437.







\bibitem[De]{delzant}
T.\ Delzant, \emph{Hamiltoniens p\'{e}riodiques et image convexe de
l'application moment}, Bulletin de la Soci\'{e}t\'{e}
Math\'{e}matique de France {\bf 116} (1988), 315--339.



\bibitem[Fu]{fulton}
W.\ Fulton, \emph{Introduction to toric varieties}, Princeton
University Press, 1993.










\bibitem[Gr]{gromovcurves}
M.\ Gromov, \emph{Pseudo holomorphic curves in symplectic
manifolds}, Inv.\ Math.\ {\bf 82} (1985), 307--347.










\bibitem[GS]{g-s:convexity}
V.\ Guillemin and S.\ Sternberg, \emph{Convexity properties of the
moment mapping}, Invent.\ Math.\ {\bf 67} (1982), 491--513.










\bibitem[HLS]{hls}
H.\ Hofer, V.\ Lizan, J.\ C.\ Sikorav, 
\emph{On genericity for holomorphic curves in four-dimensional almost-complex manifolds},
J.\ Geom.\  Anal.\  {\bf  7} (1997), no.\ 1, 149--159.





 
\bibitem[Ka1]{karshon}
Y.\ Karshon, \emph{Periodic Hamiltonian flows on four dimensional
manifolds}, Memoirs Amer.\ Math.\ Soc.\ {\bf 672}, 1999.



\bibitem[Ka2]{YK:max_tori}
Y.\ Karshon, \emph{Maximal tori in the symplectomorphism groups
of Hirzebruch surfaces}, Math.\ Research Letters {\bf 10} (2003), no.\ 1, 
125--132.


\bibitem[Ke]{thesis}
L.\ Kessler, \emph{Torus actions on small blow ups of $\CP^2$}, 
Ph.D. thesis, The Hebrew University of Jerusalem, 2004.

\bibitem[KK]{kk}
Y.\ Karshon and L.\ Kessler, \emph{Circle and torus actions
on equal symplectic blow-ups of $\CP^2$}, to appear at Mathematical Research Letters 14 (2007).
arXiv:math.SG/0501011.


\bibitem[KKP]{finite}
Y.\ Karshon, L.\ Kessler, and M.\ Pinsonnault,
\emph{A compact symplectic four-manifold admits only finitely many inequivalent toric actions},
accepted by the Journal of Symplectic Geometry (2007).
arXiv:math.SG/0609043. 














\bibitem[McD1]{McDuff-Structure} 
D. McDuff, \emph{The structure of rational and ruled symplectic
$4$-manifolds}, J.\ Amer.\ Math.\ Soc.\ \textbf{3} (1990), no.~3, 679--712.
MR1049697



\bibitem[McD2]{McD:topology}
D. McDuff, \emph{Blow ups and symplectic embeddings in dimension 4},
Topology \textbf{30}, no.~3, p.409--421, 1991.







\bibitem[MP]{mystery}
D.\ McDuff and L.\ Polterovich, \emph{Symplectic packings and
algebraic geometry}, Invent.\ Math.\ \textbf{115} (1994), no.~3, 405--434.


\bibitem[MS1]{intro}
D.\ McDuff and D.\ Salamon,
\emph{Introduction to symplectic topology}, Oxford University Press, 1998.


\bibitem[MS2]{nsmall}
D.\ McDuff and D.\ Salamon, \emph{J-holomorphic curves and symplectic 
topology}, Amer.\ Math.\ Soc.\ 2004.

\bibitem[P]{martin}
M.\  Pinsonnault, \emph{Maximal compact tori in the Hamiltonian groups of 4-dimensional symplectic manifolds}, 
 arXiv:math/0612565, 2006.  


\bibitem[W]{W}
A.\ Weinstein, \emph{Symplectic manifolds and their Lagrangian submanifolds},
Adv.\ Math.\ {\bf 6} (1971), 329--346.
\end{thebibliography}
\end{document}